\documentclass{amsart}

\usepackage{amsmath,amssymb,verbatim}
\usepackage{amsthm}
\usepackage{stmaryrd}
\usepackage{enumerate}
\usepackage{url}
\usepackage{cleveref}
\usepackage{lpic}

\title{Forking and dividing in Henson graphs}

\author{Gabriel Conant}

\address{Department of Mathematics, Statistics and Computer Science\\
University of Illinois at Chicago\\
Chicago, IL, 60607, USA}

\email{gconan2@uic.edu}

\newtheorem{theorem}{Theorem}[section]
\newtheorem{proposition}[theorem]{Proposition}
\newtheorem{lemma}[theorem]{Lemma}
\newtheorem{corollary}[theorem]{Corollary}
\newtheorem*{theorem*}{Theorem}

\theoremstyle{definition}
\newtheorem{definition}[theorem]{Definition}

\newtheorem{construction}[theorem]{Construction}
\newtheorem{fact}[theorem]{Fact}

\newtheorem{question}[theorem]{Question}

\newtheorem{fact/def}[theorem]{Fact/Definition}

\newcommand{\claim}{$_{\sslash}$}
\newcommand{\pR}{\makebox[.14in]{$R$}}
\newcommand{\nR}{\makebox[.14in]{$\not\!\! R$}}
\newcommand{\nrhd}{\!\not\!\rhd}

\def\abar{\bar{a}}

\def\bbar{\bar{b}}
\def\cbar{\bar{c}}
\def\cH{\mathcal{H}}
\def\cL{\mathcal{L}}
\def\dbar{\bar{d}}

\def\func{\longrightarrow}
\def\G{\mathbb{G}}
\def\H{\mathbb{H}}
\def\id{\operatorname{id}}
\def\M{\mathbb{M}}
\def\NIP{\operatorname{NIP}}
\def\NSOP{\operatorname{NSOP}}
\def\NTP{\operatorname{NTP}}
\def\seq{\subseteq}
\def\SOP{\operatorname{SOP}}
\def\Th{\operatorname{Th}}
\def\TP{\operatorname{TP}}
\def\tp{\operatorname{tp}}
\def\vphi{\varphi}
\def\xbar{\bar{x}}
\def\ybar{\bar{y}}

\def\Ind{\setbox0=\hbox{$x$}\kern\wd0\hbox to 0pt{\hss$\mid$\hss}
\lower.9\ht0\hbox to 0pt{\hss$\smile$\hss}\kern\wd0}

\def\Notind{\setbox0=\hbox{$x$}\kern\wd0\hbox to 0pt{\mathchardef
\nn=12854\hss$\nn$\kern1.4\wd0\hss}\hbox to
0pt{\hss$\mid$\hss}\lower.9\ht0 \hbox to 0pt{\hss$\smile$\hss}\kern\wd0}

\def\ind{\mathop{\mathpalette\Ind{}}}

\def\nind{\mathop{\mathpalette\Notind{}}}

\begin{document}

\begin{abstract}
For $n\geq 3$, define $T_n$ to be the theory of the generic $K_n$-free graph, where $K_n$ is the complete graph on $n$ vertices. We prove a graph theoretic characterization of dividing in $T_n$, and use it to show that forking and dividing are the same for complete types. We then give an example of a forking and nondividing formula. Altogether, $T_n$ provides a counterexample to a recent question of Chernikov and Kaplan. 
\end{abstract}

\maketitle

\section{Introduction}

Classification in model theory, beginning with stability theory, is strongly fueled by the study of abstract notions of independence, the frontrunners of which are forking and dividing. These notions have proved useful in the abstract treatment of independence and dimension in the stable setting, and initiated a quest to understand when they are useful in the unstable context. Significant success was achieved in the class of simple theories (see \cite{KiPi}). Meaningful results have also been found for $\NIP$ theories and, more generally, $\NTP_2$ theories, which include both simple and $\NIP$. A notable example is the following recent result from \cite{ChKa}.

\begin{theorem*}\label{ChKa}
Suppose $\M$ is a sufficiently saturated monster model of an $\NTP_2$ theory. Given $C\subset\M$, the following are equivalent.
\begin{enumerate}[$(i)$]
\item A partial type forks over $C$ if and only if it divides over $C$.
\item $C$ is an \textbf{extension base for nonforking}, i.e. if $\pi(\xbar)$ is a partial type with parameters from $C$, then $\pi(\xbar)$ does not fork over $C$.
\end{enumerate}
\end{theorem*}

In general, if condition $(i)$ holds for a set $C$, then condition $(ii)$ does as well. In fact, condition $(ii)$ should be thought of as the minimal requirement for nonforking to be meaningful for types over $C$. In particular, if $C$ is \textit{not} an extension base for nonforking, then there are types with no nonforking extensions. There are few examples where condition $(ii)$ fails, and most of them do so by exploiting some kind of circular ordering (see e.g \cite[Exercise 7.1.6]{TeZi}). On the other hand, condition $(i)$ is, \textit{a priori}, harder to achieve. It is useful because it allows us to ignore the subtlety of forking versus dividing. However, in every textbook example where condition $(i)$ fails, it is because condition $(ii)$ also fails. This leads to the natural question, which is asked in \cite{ChKa}, of whether the result above extends to classes of theories other than $\NTP_2$. In this paper, we give an example of an $\NSOP_4$ theory in which condition $(ii)$ holds for all sets, but condition $(i)$ fails.

We will consider forking and dividing in the theory of a well-known structure: the generic $K_n$-free graph, also known as the \textit{Henson graph}, a theory with $\TP_2$ and $\NSOP_4$. Our main goal is to characterize forking and dividing in the theory of the Henson graph. We will show that dividing independence has a meaningful graph-theoretic interpretation, and has something to say about the combinatorics of the structure. Using this characterization, we will show that despite the complexity of the theory, forking and dividing are the same for complete types. As a consequence, every set is an extension base for nonforking, and so nonforking/nondividing extensions always exist. On the other hand, we will show that there are formulas which fork, but do not divide.

\subsection*{Acknowledgements} I would like to thank Lynn Scow, John Baldwin, Artem Chernikov, Dave Marker, Caroline Terry, and Phil Wesolek for their part in the development of this project.

\section{Model Theoretic Preliminaries} \label{model theory}

This section contains the definitions and basic facts concerning forking and dividing. We first specify some conventions that will be maintained throughout the paper. If $T$ is a complete first order theory and $\M$ is a monster model of $T$, we write $A\subset\M$ to mean that $A$ is a ``small" subset of $\M$, i.e. $A\seq\M$ and $\M$ is $|A|^+$-saturated. We use the letters $a,b,c,\ldots$ to denote singletons, and $\abar,\bbar,\cbar,\ldots$ to denote tuples (of possibly infinite length).

\begin{definition} Suppose $C\subset\M$, $\pi(\xbar,\ybar)$ is a partial type with parameters from $C$, and $\bbar\in\M$. 
\begin{enumerate}
\item $\pi(\xbar,\bbar)$ \textbf{divides over $C$} if there is a $C$-indiscernible sequence $(\bbar^l)_{l<\omega}$, with $\bbar^0=\bbar$, such that $\bigcup_{l<\omega}\pi(\xbar,\bbar^l)$ is inconsistent.
\item $\pi(\xbar,\bbar)$ \textbf{forks over $C$} if there is some $D\supseteq \bbar C$ such that if $p\in S_n(D)$ extends $\pi(\xbar,\bbar)$ then $p$ divides over $C$. 
\item A formula $\vphi(\xbar,\bbar)$ \textbf{forks} (resp. \textbf{divides}) over $C$ if $\{\vphi(\xbar,\bbar)\}$ forks (resp. divides) over $C$.
\end{enumerate}
\end{definition}

The following basic facts can be found in \cite{TeZi}.

\begin{proposition}\label{facts} Let $C\subset\M$.
\begin{enumerate}[(a)]
\item If a complete type forks (resp. divides) over $C$ then it contains some formula that forks (resp. divides) over $C$.
\item If $\pi(\xbar)$ is a consistent type over $C$ then $\pi(\xbar)$ does not divide over $C$.
\end{enumerate}
\end{proposition}

Nondividing and nonforking are used to define ternary relations on small subsets of $\M$, given by
\begin{enumerate}
\item $A\ind^d_C B$ if and only if $\tp(A/BC)$ does not divide over $C$,
\item $A\ind^f_C B$ if and only if $\tp(A/BC)$ does not fork over $C$.
\end{enumerate}

These relations were originally defined to abstractly capture notions of independence and dimension in stable theories, and have been found to still be meaning ful in more complicated theories as well. In particular, we will consider the interpretation of these notions in the unstable theories of certain homogeneous graphs.

\section{Graphs}\label{graphs}

Recall that a countable graph $G$ is \textbf{universal} if any countable graph is isomorphic to an induced subgraph of $G$; and $G$ is \textbf{homogenous} if any graph isomorphism between finite subsets of $G$ extends to an automorphism of $G$. 

The canonical example of such a graph is the countable \textit{random graph}, i.e. the Fra\"{i}ss\'{e} limit of the class of finite graphs. In \cite{Henson}, a new family of countable homogenous graphs was introduced: the generic $K_n$-free graphs, for $n\geq 3$, which are often called the \textit{Henson graphs}. For a particular $n\geq 3$, there is a unique such graph up to isomorphism.

\begin{definition}
Fix $n\geq 3$ and let $K_n$ be the complete graph on $n$ vertices. The generic $K_n$-free graph, $\cH_n$, is the unique countable graph such that
\begin{enumerate}[$(i)$]
\item $\cH_n$ is $K_n$-free,
\item any finite $K_n$-free graph is isomorphic to an induced subgraph of $\cH_n$,
\item any graph isomorphism between finite subsets of $\cH_n$ extends to an automorphism of $\cH_n$.
\end{enumerate}
\end{definition}

Given $n\geq 3$, $\cH_n$ can also be constructed as the Fra\"{i}ss\'{e} limit of the clas of finite $K_n$-free graphs.

We study graph structures in the graph language $\cL=\{R\}$, where $R$ is interpreted as the binary edge relation. We consider
\begin{enumerate}
\item $T_0$, the complete theory of the random graph,
\item $T_n=\Th(\cH_n)$, for $n\geq 3$.
\end{enumerate}
It is a well-known fact (and a standard exercise) that each of these theories is $\aleph_0$-categorical with quantifier elimination.

Fix $n\geq 3$ and fix $\H_n\models T_n$, a sufficiently saturated ``monster" model of $T_n$. As $\H_n$ is a graph, we can embed it in a \textit{larger} sufficiently saturated ``monster" model $\G\models T_0$. Note that $\H_n$ is then a subgraph of $\G$, but not an elementary substructure. Let $\kappa(\H_n)=\sup\{\kappa:\textnormal{$\H_n$ is $\kappa$-saturated}\}$.

For the rest of the paper, $n$, $\H_n$, and $\G$ are fixed. By saturation, we have the following fact.

\begin{proposition}\label{gext}
Suppose $C\subset\H$ and $X\seq\G$, such that $X$ is $K_n$-free, $C\seq X$, and $|X|\leq\kappa(\H)$. Then there is a graph embedding $f:X\func\H$ such that $f|_C=\id_C$.
\end{proposition}

The remainder of this section is devoted to specifying notation and conventions concerning the language $\cL$. First, we consider types. 

\begin{definition} Suppose $C\subset\G$, with $|C|<\kappa(\H_n)$.
\begin{enumerate}
\item We only consider partial types $\pi(\xbar)$ such that $|\xbar|\leq\kappa(\H_n)$. Furthermore, we will assume types are ``symmetricially closed". For example if $c\in C$ then $x\pR c\in\pi(\xbar)$ if and only if $c\pR x\in\pi(\xbar)$.
\item An \textbf{$R$-type over $C$} is a collection $\pi(\xbar)$ of atomic and negated atomic $\cL$-formulas, none of which is of the form $x_i=c$, where $c\in C$. When we say that $\pi(\xbar,\ybar)$ is an $R$-type over $C$, we will assume further that $\pi(\xbar,\ybar)$ does not contain $x_i=x_j$ or $y_i=y_j$, for some $i\neq j$.
\item Suppose $\pi(\xbar)$ is an $R$-type over $C$. An \textbf{optimal solution} of $\pi(\xbar)$ is a tuple $\abar\models\pi(\xbar)$ such that 
\begin{enumerate}[$(i)$]
\item $a_i\neq a_j$ for all $i\neq j$ and $a_i\not\in C$ for all $i$,
\item $a_i\pR a_j$ if and only if $x_i\pR x_j\in\pi(\xbar)$,
\item given $c\in C$, $a_i\pR c$ if and only if $x_i\pR c\in\pi(\xbar)$.
\end{enumerate}
\end{enumerate}
\end{definition}

We will frequently use the following fact, which says that we can always find optimal solutions of $R$-types.

\begin{proposition}
Suppose $C\subset\H_n$ and $\pi(\xbar)$ is an $R$-type over $C$.
\begin{enumerate}[$(a)$]
\item $\pi(\xbar)$ is consistent with $T_0$ if and only if it has an optimal solution in $\G$.
\item $\pi(\xbar)$ is consistent with $T_n$ if and only if it has an optimal solution in $\H_n$.
\end{enumerate}
\end{proposition}

This is a straightforward exercise, which we leave to the reader. The idea is that a type cannot prove than an edge exists in a graph, without explicitly saying so. Moreover, removing extra edges to ``optimize" the solution of a consistent type is always possible and, in the case of $T_n$, will not conflict with the requirement that the solution be $K_n$-free.

Next, we specify notation and conventions concerning $\cL$-formulas.

\begin{definition} Suppose $C\subset\G$.
\begin{enumerate}
\item Let $\cL_0(C)$ be the collection of conjunctions of atomic and negated atomic $\cL$-formulas, with parameters from $C$, such that no conjunct is of the form $x=c$, where $x$ is a variable and $c\in C$. When we write $\vphi(\xbar,\ybar)\in\cL_0(C)$, we will assume further that no conjunct of $\vphi(\xbar,\ybar)$ is of the form $x_i=x_j$ or $y_i=y_j$, for some $i\neq j$.
\item Given $\vphi(\xbar)\in\cL_0(C)$ and $\theta(\xbar)$, an atomic or negated atomic formula, we write ``$\vphi(\xbar)\rhd\theta(\xbar)$" if $\theta(\xbar)$ is a conjunct of $\vphi(\xbar)$.
\item We will assume $\cL_0(C)$-formulas are ``symmetrically closed". For example $\vphi(\xbar)\rhd x\pR c$ if and only if $\vphi(\xbar)\rhd c\pR x$.
\item Let $\cL_R(C)$ be the collection of formulas $\vphi(\xbar,\ybar)\in\cL_0(C)$ such that no conjunct is of the form $x_i=y_j$.
\end{enumerate}
\end{definition}

The main result of this paper will be a characterization of forking and dividing in $T_n$. We will use the following characterization of dividing in $T_0$, which is a  standard exercise (see e.g. \cite{TeZi}).

\begin{fact}\label{T0 forking}
Fix $C\subset\G$ and $\vphi(\xbar,\ybar)\in\cL_0(C)$. Suppose $\bbar\in\G\backslash C$ is such that $\vphi(\xbar,\bbar)$ is consistent. Then $\vphi(\xbar,\bbar)$ divides over $C$ if and only if $\vphi(\xbar,\bbar)\rhd x_i=b$ for some $b\in\bbar$. Consequently, if $A,B,C\subset\G$ then $A\ind^d_C B\Leftrightarrow A\cap B\seq C$. 
\end{fact}

$T_0$ is a standard example of a \textit{simple theory}, and so the previous fact is also a characterization of forking. On the other hand, $T_n$ is non-simple. Indeed, the Henson graph is a canonical example where $\ind^f$ fails \textit{amalgamation over models} (see \cite{KiPi}). A direct proof of this (for $n=3$) can be found in \cite[Example 2.11(4)]{Hartstab}. The precise classification of $T_n$ is well-known, and summarized by the following result.

\begin{fact} 
$T_n$ is $\TP_2$, $\SOP_3$, and $\NSOP_4$.
\end{fact}

See \cite{ChNTP2} and \cite{Sh500} for definitions of these properties. The proof of $\TP_2$ can be found in \cite{ChNTP2} for $n=3$. $\SOP_3$ and $\NSOP_4$ are demonstrated in \cite{Sh500} for $n=3$. The generalizations of these arguments to $n\geq 3$ are fairly straightforward. However, $\NSOP_4$ for all $n\geq 3$ also follows from a more general result in \cite{PaSOP4}.

\section{Dividing in $T_n$} \label{divsec}

The goal of this section is to find a graph theoretic characterization of dividing independence in $T_n$. Therefore, when we say that a partial type divides over $C\subset\H_n$, we mean with respect to the theory $T_n$.

\begin{lemma}\label{no vertical} 
Suppose $(\bbar^l)_{l<\omega}$ is an indiscernible sequence in $\H_n$.
\begin{enumerate}[$(a)$]
\item $\H_n\models\neg b^k_i\pR b^l_i$ for all $k<l<\omega$ and $1\leq i\leq |\bbar^0|$.
\item If $l(\bbar^0)<n-1$ then $\bigcup_{l<\omega}\bbar^l$ is $K_{n-1}$ free.
\end{enumerate}
\end{lemma}
\begin{proof}
Part $(a)$. Suppose not. By indiscernibility, $\H_n\models b^k_i\pR b^l_i$ for all $k<l<\omega$. In particular, $\{b^1_i,\ldots,b^n_i\}\cong K_n$, which is a contradiction.

Part $(b)$. Suppose $S\seq B:=\bigcup_{l<\omega}\bbar^l$ with $|S|=n-1$ and, for $1\leq i\leq |\bbar^0|$, let $S_i=S\cap\{b^l_i:l<\omega\}$. Since $l(\bbar^0)<n-1$, there is some $i$ such that $|S_i|\geq 2$. By part (a), $S\not\cong K_{n-1}$.
\end{proof}

We first define a graph theoretic binary relation on disjoint graphs, which will capture the notion of dividing in $T_n$.

\begin{definition}\label{bound def}$~$
\begin{enumerate}
\item Suppose $B,C\subset\H_n$ are disjoint. Then $B$ is \textbf{$n$-bound to $C$}, written $K_n(B/C)$, if there is $B_0\seq BC$ such that
\begin{enumerate}[$(i)$]
\item $|B_0|=n$ and $B_0\cap C\neq\emptyset\neq B_0\cap B$,
\item if $u,v\in B_0\cap C$ are distinct then $u\pR v$,
\item if $u\in B_0\cap B$ and $v\in B_0\cap C$ then $u\pR v$.
\end{enumerate}
We say $B_0$ \textbf{witnesses} $K_n(B/C)$. Informally, $B_0$ witnesses $K_n(B/C)$ if and only if the only thing preventing $B_0\cong K_n$ is a possible lack of edges between vertices in $B$.
\item Suppose $\vphi(\xbar,\ybar)\in\cL_R(C)$ and $\bbar\in\H_n\backslash C$ such that $\vphi(\xbar,\bbar)$ is consistent. Then $\bbar$ is \textbf{$\vphi$-$n$-bound to $C$}, written $K^\vphi_n(\bbar/C)$, if there is $B\seq\bbar$, with $0<|B|<n$, such that
\begin{enumerate}[$(i)$]
\item $\neg K_n(B/C)$,
\item $K_n(B/\abar C)$ for all $\abar\models\vphi(\xbar,\bbar)$.
\end{enumerate} 
We say $B$ \textbf{witnesses} $K^\vphi_n(\bbar/C)$.
\end{enumerate}
\end{definition}

The main result of this section (Theorem \ref{Tn dividing}) will show that $K^\vphi_n$ is the graph theoretic interpretation of dividing. In particular, for $\vphi(\xbar,\ybar)\in\cL_R(C)$ and $\bbar\in\H_n\backslash C$ with $\vphi(\xbar,\bbar)$ consistent, we will show that $\vphi(\xbar,\bbar)$ divides over $C$ if and only if $K_n^\vphi(\bbar/C)$. The reverse direction of the proof of this will use the following construction of a particular indiscernible sequence. 

\begin{construction}\label{Gamma def}
Fix $C\subset\H_n$ and $\bbar\in\H_n\backslash C$. We extend $\bbar$ to an infinite $C$-indiscernible sequence $(\bbar^l)_{l<\omega}$, such that $b^l_i\neq b^m_j$ and $\neg b^l_i\pR b^m_j$ for all $l<m<\omega$ and $1\leq i,j\leq|\bbar|$. Note that $(\bbar^l)_{l<\omega}$ is $K_n$-free and so $(\bbar^l)_{l<\omega}$ is an indiscernible sequence in $\H_n$.

Given $B\seq\bbar$, we let $\Gamma(C\bbar,B)$ be the graph expansion of $C\cup(\bbar^l)_{l<\omega}$ obtained by adding $b^l_i\pR b^m_j$ if and only if $l<m$, $i<j$, and $b_i,b_j\in B$. We can embed $\Gamma(C\bbar,B)$ into $\G$ over $C$. Furthermore, if $\Gamma(C\bbar,B)$ is $K_n$-free, then we can embed $\Gamma(C\bbar,B)$ in $\H_n$ over $C$. In this case, if $\Gamma_0(C\bbar,B)$ is the image of $(\bbar^l)_{l<\omega}$, then $\Gamma_0(C\bbar,B)$ is a $C$-indiscernible sequence in $\H_n$. 
\end{construction}

\begin{lemma}\label{Gamma}
Let $C\subset\H_n$ and $\bbar\in\H_n\backslash C$. Suppose $\vphi(\xbar,\ybar)\in\cL_R(C)$ such that $\vphi(\xbar,\bbar)$ is consistent and $K_n^\vphi(\bbar/C)$, witnessed by $B\seq\bbar$. 
\begin{enumerate}[$(a)$]
\item $\Gamma(C\bbar,B)$ is $K_n$-free.
\item If $\Gamma_0(C\bbar,B)=(\bbar^l)_{l<\omega}$ then $\{\vphi(\xbar,\bbar^l):l<\omega\}$ is $(n-1)$-inconsistent with $T_n$.
\end{enumerate}
\end{lemma}
\begin{proof}
We may consider $\Gamma_0(C\bbar,B)$ as an indiscernible sequence in $\G$. 

Part $(a)$. Suppose $K_n\cong W\seq\Gamma(C\bbar,B)$. Since $C$ is $K_n$-free, $W\cap\Gamma_0(C\bbar,B)\neq\emptyset$. Say $W\cap\Gamma_0(C\bbar,B)=\{b^{l_1}_{i_1},\ldots,b^{l_r}_{i_r}\}$ with $l_1\leq\ldots\leq l_r$.  Note that $i_s\neq i_t$ for all $1\leq s<t\leq r$ by \Cref{no vertical}. Let $B_0=\{b_{i_1},\ldots,b_{i_r}\}$. Define
$$
V=(W\backslash\{b^{l_1}_{i_1},\ldots,b^{l_r}_{i_r}\})\cup\{b_{i_1},\ldots,b_{i_r}\}.
$$
\indent If $l_1=l_r$ then, since $\bbar^{l_1}\equiv_C\bbar$, it follows that $V\cong K_n$, which is a contradiction. Therefore $l_1<l_r$. By construction of $\Gamma(C\bbar,B)$ it follows that $b_{i_1},b_{i_r}\in B$. If $1\leq s\leq r$ then we either have $l_1<l_s$ or $l_s<l_r$, and in either case it follows that $b_{i_s}\in B$. Therefore $r\leq |B|\leq n-1$; in particular $C\cap W\neq \emptyset$. But then $V$ witnesses that $B$ is $n$-bound to $C$, which is a contradiction. Therefore $\Gamma(C\bbar,B)$ is $K_n$-free.

Part $(b)$. By part $(a)$, we may consider $\Gamma_0(C\bbar,B)$ as in indiscernible sequence in $\H_n$. By indiscernibility, it suffices to show that the $R$-type $\pi(\xbar)=\{\vphi(\xbar,\bbar^l):0\leq l<n-1\}$ is inconsistent with $T_n$. So suppose, towards a contradiction, that $\pi(\xbar)$ is consistent with $T_n$ and let $\abar\in\H_n$ be an optimal solution. Then $\abar\models\vphi(\xbar,\bbar)$ so, by assumption, there is $D\seq BC\abar$ witnessing $K_n(B/C\abar)$. We have $D\cap B\neq\emptyset$. Moreover, $\neg K_n(B/C)$ implies $D\cap\abar\neq\emptyset$. To ease notation, let
$$
D\cap B=\{b_0,\ldots,b_k\}\text{ and }D\cap \abar=\{a_0,\ldots,a_m\}.
$$
Note that $k<n-1$. Define $A_0=D\cap C\abar$ and $B_0=\{b^0_0,\ldots,b^k_k\}$. We make the following observations.
\begin{enumerate}
\item If $u,v\in A_0$ are distinct then $u\pR v$.
\item If $b^i_i,b^j_j\in B_0$ are distinct, then $b^i_i\pR b^j_j$ by construction of $\Gamma(C\bbar,B)$.
\item If $c\in A_0\cap C$ and $b^j_j\in B_0$ then $b^j_j\pR c$ since $b_j\pR c$ and $\bbar^j\equiv_C\bbar$.
\item If $a_i\in A_0\cap\abar$ and $b^j_j\in B_0$ then, since $\abar$ is an optimal solution of $\pi(\xbar)$, we have
$$
a_i\pR b_j\Rightarrow\left(\vphi(\xbar,\bbar)\rhd x_i\pR b_j\right)\Rightarrow\left(\vphi(\xbar,\bbar^j)\rhd x_i\pR b^j_j\right)\Rightarrow a_i\pR b^j_j.
$$
\end{enumerate}
These observations imply $A_0B_0\cong K_n$, which is a contradiction since $A_0B_0\subset\H$.
\end{proof}

\begin{theorem}\label{Tn dividing}
Let $C\subset\H_n$, $\vphi(\xbar,\ybar)\in\cL_R(C)$, and $\bbar\in\H_n\backslash C$ such that $\vphi(\xbar,\bbar)$ is consistent. Then $\vphi(\xbar,\bbar)$ divides over $C$ if and only if $K^\vphi_n(\bbar/C)$.
\end{theorem}
\begin{proof}
$(\Leftarrow)$: Suppose $B\seq\bbar$ witnesses $K^\vphi_n(\bbar/C)$. Then $\Gamma(C\bbar,B)\seq\H_n$ and $\{\vphi(\xbar,\bbar^l):l<\omega\}$ is $(n-1)$-inconsistent by \Cref{Gamma}. So $\vphi(\xbar,\bbar)$ divides over $C$.

$(\Rightarrow)$: Suppose $\vphi(\xbar,\bbar)$ divides over $C$. Then there is a $C$-indiscernible sequence $(\bbar^l)_{l<\omega}$, with $\bbar^0=\bbar$, such that $\{\vphi(\xbar,\bbar^l):l<\omega\}$ is inconsistent.

Let $G=C\cup\bigcup_{l<\omega}\bbar^l$. Consider $G$ as a subgraph of $\G$, and note that $(\bbar^l)_{l<\omega}$ is still $C$-indiscernible in $\G$. Since $\vphi(\xbar,\ybar)\in\cL_R(C)$ and $\vphi(\xbar,\bbar)$ is consistent (in $\G$), it follows from \Cref{T0 forking} that $\vphi(\xbar,\bbar)$ does not divide over $C$ in $\G$. Therefore there is an optimal realization $\dbar\in\G$ of $\pi(\xbar):=\{\vphi(\xbar,\bbar^l):l<\omega\}$. If $G\dbar$ is $K_n$-free then $G\dbar$ embeds in $\H_n$ over $G$, which is a contradiction. Therefore there is $K_n\cong W\seq G\dbar$.  Note that $W\cap\dbar\neq\emptyset$ since $G$ is $K_n$-free. Without loss of generality, let $W\cap\dbar=\{d_1,\ldots,d_m\}$. 

Suppose, towards a contradiction, that $W\cap G\seq C$. Let $\abar\in\H_n$ be a solution to $\vphi(\xbar,\bbar)$. Since $\dbar$ is an optimal realization of $\pi(\xbar)$, we make the following observations.
\begin{enumerate}
\item If $1\leq i,j\leq m$ are distinct then $a_i\pR a_j$.
\item If $b\in W\cap G$ then $a_i\pR b$ for all $1\leq i\leq m$.
\end{enumerate}
Therefore, $K_n\cong(W\cap G)\cup\{a_1,\ldots,a_m\}$, which is a contradiction, and so $W\cap(G\backslash C)\neq\emptyset$.

Let $W\cap(G\backslash C)=\{b^{l_1}_{j_1},\ldots,b^{l_k}_{j_k}\}$ and note that $1\leq t<n$. By Lemma \ref{no vertical}, $j_s\neq j_t$ for all $s\neq t$, so without loss of generality, let $W\cap(G\backslash C)=\{b^{l_1}_1,\ldots,b^{l_k}_k\}$. Let $B=\{b_1,\ldots,b_k\}$.

\noindent\textit{Claim 1}: $\neg K_n(B/C)$.

\noindent\textit{Proof}: Suppose $X\seq BC$ witnesses $K_n(B/C)$. By indiscernibility, if $B_0=\{b^{l_s}_s:b_s\in B\cap X\}$, then $(X\cap C)\cup B_0$ witnesses that $B_0$ is $n$-bound to $C$. But $b^{l_s}_s\pR b^{l_t}_t$ for all $s\neq t$, so $(X\cap C)\cup B_0\cong K_n$, which is a contradiction.\claim

\noindent\textit{Claim 2}: If $\abar\in\H_n$ is a solution of $\vphi(\xbar,\bbar)$ then $K_n(B/C\abar)$.

\noindent\textit{Proof}: We show $(W\cap C)\cup B\cup\{a_1,\ldots,a_m\}$ witnesses $K_n(B/C\abar)$, which means verifying all of the necessary relations. Recall that $\dbar$  is an optimal solution to $\pi(\xbar)$. 
\begin{enumerate}
\item If $1\leq i\neq j\leq m$ then $d_i\pR d_j\Rightarrow\left(\vphi(\xbar,\bbar)\rhd x_i\pR x_j\right)\Rightarrow a_i\pR a_j$.
\item If $1\leq i\leq m$ and $c\in W\cap C$ then $d_i\pR c\Rightarrow\left(\vphi(\xbar,\bbar)\rhd x_i\pR c\right)\Rightarrow a_i\pR c$.
\item If $1\leq i\leq m$ and $1\leq s\leq k$ then $d_i\pR b^{l_s}_s\Rightarrow\left(\vphi(\xbar,\bbar^{l_s})\rhd x_i\pR b^{l_s}_s\right)\Rightarrow\left(\vphi(\xbar,\bbar )\rhd x_i\pR b_s\right)\Rightarrow a_i\pR b_s$.\claim
\end{enumerate}

Together, Claims 1 and 2 imply $K^\vphi_n(\bbar/C)$, as desired. 
\end{proof}

We can now give the full characterization of nondividing formulas in $T_n$, and the ternary relation $\ind^d$ on sets, which gives the analogy of Fact \ref{T0 forking} for $T_n$.

\begin{theorem}\label{div ind n}$~$
\begin{enumerate}[$(a)$]
\item Suppose $C\subset\H_n$, $\vphi(\xbar,\ybar)\in\cL_0(C)$, and $\bbar\in\H_n\backslash C$ such that $\vphi(\xbar,\bbar)$ is consistent. Then $\vphi(\xbar,\bbar)$ divides over $C$ if and only if $\vphi(\xbar,\bbar)\rhd x_i=b$ for some $b\in\bbar$, or $\vphi(\xbar,\ybar)\in\cL_R(C)$ and $K^\vphi_n(\bbar/C)$.
\item Suppose $A,B,C\subset\H_n$. Then $A\ind^d_C B$ if and only if 
\begin{enumerate}[$(i)$]
\item $A\cap B\seq C$, and
\item for all $\bbar\in B\backslash C$, if $K_n(\bbar/AC)$ then $K_n(\bbar/C)$.
\end{enumerate}
\end{enumerate}
\end{theorem}
\begin{proof}
Part $(a)$ follows immediately from Theorem \ref{Tn dividing}. For part $(b)$, let $\abar$ enumerate $A$. Note that $\abar$ is an optimal solution to $\tp(\abar/BC)$. 

$(\Rightarrow)$: Suppose $A\ind^d_C B$. Then $A\cap B\seq C$ is immediate. For condition $(ii)$, fix $\bbar\in B\backslash C$ and suppose, towards a contradiction, that $\neg K_n(\bbar/C)$ and $K_n(\bbar/AC)$. Fix $X\seq AC\bbar$ witnessing $K_n(\bbar/AC)$. Let $C_0=X\cap C$. Without loss of generality, let $X\cap A=\{a_1,\ldots,a_m\}$ and $X\cap \bbar=\bbar_*=\{b_1,\ldots,b_k\}$. Define $\cL_R(C)$-formula $\vphi(\xbar,\ybar)$, where $x_i=(x_1,\ldots,x_m)$, $\ybar=(y_1,\ldots,y_k)$, and $\vphi(\xbar,\ybar)$ expresses
\begin{enumerate}[$\bullet$]
\item $\xbar$ is a complete graph and $\xbar\pR\ybar$,
\item $\xbar\pR C_0$ and $\ybar\pR C_0$,
\item $C_0$ is a complete graph.
\end{enumerate}
Then $K_n^\vphi(\bbar_*/C)$ and $\vphi(\xbar,\bbar_*)\in\tp(A/BC)$. Therefore $A\nind^d_C B$ by Theorem \ref{Tn dividing}, which is a contradiction.

$(\Leftarrow)$: Suppose $A\nind^d_C B$. Then there is some $\vphi(\xbar,\ybar)\in\cL_0(C)$ and $\bbar\in B\backslash C$ such that $\vphi(\xbar,\bbar)$ divides over $C$ and $\vphi(\xbar,\bbar)\in\tp(A/BC)$. If $\vphi(\xbar,\ybar)\rhd x_i=y_j$ for some $i,j$, then $a_i=b_j\in(A\cap B)\backslash C$ and $(i)$ fails.  Otherwise, $\vphi(\xbar,\ybar)\in\cL_R(C)$ and  $K^\vphi_n(\bbar/C)$. It follows that $\neg K_n(\bbar/C)$ and $K_n(\bbar/AC)$, so $(ii)$ fails.
\end{proof}

The theorem translates the model theoretic notion of dividing to the graph theoretic notion $K^\vphi_n(\bbar/C)$. Although the definition of $K^\vphi_n(\bbar/C)$  implies that we must check all solutions of $\vphi$, it suffices to check an optimal one.

\begin{corollary}\label{Tn dividing 2}
Let $C\subset\H_n$, $\vphi(\xbar,\ybar)\in\cL_R(C)$, and $\bbar\in\H_n\backslash C$ such that $\vphi(\xbar,\bbar)$ is consistent. Let $\abar$ be an optimal solution. Then $\vphi(\xbar,\bbar)$ divides over $C$ if and only if there is $B\seq\bbar$ such that $\neg K_n(B/C)$ and $K_n(B/C\abar)$.
\end{corollary}
\begin{proof}
By Theorem \ref{Tn dividing}, we need to show $K_n^\vphi(\bbar/C)$ if and only if there is $B\seq\bbar$ such that $\neg K_n(B/C)$ and $K_n(B/C\abar)$. The forward direction is clear.

Conversely, suppose $B\seq\bbar$ such that $\neg K_n(B/C)$ and $K_n(B/C\abar)$. Let $\dbar$ be any solution to $\vphi(\xbar,\bbar)$. We want to show $K_n(B/C\dbar)$. Let $B_0\seq BC\abar$ witness $K_n(B/C\abar)$. Define $C_0=(B_0\cap C)$ and $D=\{d_i:a_i\in B_0\cap\abar\}$.
Since  $\abar$ is optimal, we can make the following observations to show that $B_0C_0D$ witnesses $K_n(B/C\dbar)$.
\begin{enumerate}
\item If $c_1,c_2\in C_0$ then $c_1\pR c_2$ by assumption.
\item If $c\in C_0$ and $d_i\in D$ then $a_i\pR c\Rightarrow\left(\vphi(\xbar,\bbar)\rhd x_i\pR c\right)\Rightarrow d_i\pR c$.
\item If $c\in C_0$ and $b\in B_0$ then $b\pR c$ by assumption.
\item If $d_i\in D$ and $b\in B_0$ then $a_i\pR b\Rightarrow\left(\vphi(\xbar,\bbar)\rhd x_i\pR b\right)\Rightarrow d_i\pR b$.\qedhere
\end{enumerate}
\end{proof}

We end this section by giving some examples and traits of dividing formulas in $T_n$. We will begin using the following notation. If $A$ and $B$ are sets we write $A\pR B$ to mean $a\pR b$ for all $a\in A$ and $b\in B$. On other hand, $A\nR B$ means $\neg a\pR b$ for all $a\in A$ and $b\in B$.

\begin{corollary}\label{dividing example}
Suppose $C\subset\H_n$ and $b_1,\ldots,b_{n-1}\in\H_n\backslash C$ are distinct. Then the formula
$$
\vphi(x,\bbar):=\bigwedge_{i=1}^{n-1}xRb_i
$$
divides over $C$ if and only if $\neg K_n(\bbar/C)$.
\end{corollary}
\begin{proof}
First, if $\bbar\cong K_{n-1}$ then $\vphi(x,\bbar)$ is inconsistent and thus divides over $C$. Furthermore, in this case $\neg K_n(\bbar/C)$ since, if so, then there is some $c\in C$ such that $c\pR\bbar$ and so $c\bbar\cong K_n$. Therefore we may assume $\bbar\not\cong K_{n-1}$.

$(\Rightarrow)$: If $\vphi(x,\bbar)$ divides over $C$ then, by Theorem \ref{Tn dividing}, there is some $B\seq\bbar$ such that $\neg K_n(B/C)$ and $K_n(B/Ca)$ for any $a\models\vphi(x,\bbar)$. Let $a\models\vphi(x,\bbar)$ such that $a\nR C$. Let $X\seq CBa$ witness $K_n(B/Ca)$. Then $\neg K_n(B/C)$ implies $a\in X$, and so $X\cap C=\emptyset$ since $a\nR C$. Therefore $X\seq Ba\seq\bbar a$, $|X|=n$, and $|\bbar|=n-1$. It follows that $B=\bbar$, and so $\neg K_n(\bbar/C)$.

$(\Leftarrow)$: Note that if $a$ realizes $\vphi(x,\bbar)$, then $K_n(\bbar/a)$. So if $\neg K_n(\bbar/C)$ then $\bbar$ itself witnesses $K^\vphi_n(\bbar/C)$. By Theorem \ref{Tn dividing}, $\vphi(x,\bbar)$ divides over $C$.
\end{proof}

\begin{corollary}\label{no R}
Let $C\subset\H$ and $\vphi(\xbar,\ybar)\in\cL_R(C)$ such that $\vphi(\xbar,\ybar)\nrhd x_i\pR y_j$ for all $i,j$. If $\bbar\in\H\backslash C$ such that $\vphi(\xbar,\bbar)$ is consistent, then $\vphi(\xbar,\bbar)$ does not divide over $C$.
\end{corollary}
\begin{proof}
Suppose $B\seq\bbar$ is such that $\neg K_n(B/C)$. Let $\abar$ be an optimal solution to $\vphi(\xbar,\bbar)$. Then $\abar\nR B$, so $\neg K_n(B/C\abar)$. By Corollary \ref{Tn dividing 2}, $\vphi(\xbar,\bbar)$ does not divide over $C$.
\end{proof}

\begin{corollary}\label{parameters n}
Let $C\subset\H$ and $\vphi(\xbar,\ybar)\in\cL_R(C)$. Suppose $\bbar\in\H\backslash C$ such that $\vphi(\xbar,\bbar)$ is consistent and divides over $C$. Define
$$
R^\vphi=\{b\in C\bbar:\vphi(\xbar,\bbar)\rhd x_i\pR b\textnormal{ for some $i$}\}\cup\{x_i:\vphi(\xbar,\bbar)\rhd x_i\pR b\textnormal{ for some $b\in C\bbar$}\}.
$$
Then $|R^\vphi|\geq n$ and $|\bbar\cap R^\vphi|>1$.
\end{corollary}
\begin{proof}
By assumption, we have $K_n^\vphi(\bbar/C)$. If $\abar$ is an optimal solution of $\vphi(\xbar,\bbar)$ then there is some $X\seq C\bbar\abar$ witnessing $K_n(\bbar/C\abar)$. Note that $X\cap\abar\neq\emptyset$ since $\neg K_n(\bbar/C)$. Set $B=(X\cap C\bbar)\cup\{x_i:a_i\in X\}$. Then $|B|\geq n$ and $B\seq R^\vphi$ since $\abar$ is optimal. Finally, $|\bbar\cap B|=|\bbar\cap X|>1$, since otherwise $X\cong K_n$. 
\end{proof}

Corollary \ref{parameters n} says that if a formula from $\cL_R(C)$ divides then it needs to mention edges between at least $n$ vertices (and more than one parameter). This is not surprising since no consistent formula from $\cL_R(C)$ will divide in $T_0$, and so dividing in $T_n$ should come from the creation of a graph that is too close to $K_n$.

\section{Forking for complete types}\label{forksec}

In this section, we use our characterization of $\ind^d$ in $T_n$ to show that forking and dividing are the same for complete types. The proof takes two steps, the first of which is to prove \textit{full existence} for the following ternary relation on graphs. We take the following definition from \cite{Adgeo}.

\begin{definition}
Given $A,B,C\subset\H_n$, define \textbf{edge independence} by
$$
\textstyle A\ind^R_C B\Leftrightarrow A\cap B\seq C\text{ and there is no edge from $A\backslash C$ to $B\backslash C$}.
$$
\end{definition}

\begin{lemma}\label{fullext}
For all $A,B,C\subset\H_n$ there is $A'\equiv_C A$ such that $A'\ind^R_C B$.
\end{lemma}
\begin{proof}
Fix $A,B,C\subset\H_n$ and enumerate $A\backslash(BC)=(a_i)_{i<\lambda}$. We define a graph $G=BC(a'_i)_{i<\lambda}\subset\G$, where each $a'_i$ is a new vertex, and
\begin{enumerate}[$(i)$]
\item for all $i,j<\lambda$, $a'_i\pR a'_j$ if and only if $a_i\pR a_j$,
\item for all $i<\lambda$ and $c\in C$, $a'_i\pR c$ if and only if $a_i\pR c$,
\item for all $i<\lambda$ and $b\in B\backslash C$, $\neg a'_iRb$.
\end{enumerate}
We claim that $G$ is $K_n$-free. Indeed, if $K_n\cong W\seq G$ then, by $(iii)$, it follows that $W\seq BC$ or $W\seq C(a'_i)_{i<\lambda}$. But $BC\subset\H_n$ so this means $W\seq C(a'_i)_{i<\lambda}$, which, by $(i)$ and $(ii)$, means $AC$ is not $K_n$-free, a contradiction. 

Therefore $G$ embeds in $\H_n$ over $BC$. Let $(a'')_{i<\lambda}$ be the image of $(a'_i)_{i<\lambda}$ and set $A'=(A\cap C)\cup(a''_i)_{i<\lambda}$. Then it is clear that $A'\equiv_C A$ and $A'\ind^R_C B$.
\end{proof}

Using full existence of $\ind^R$, we can prove the full characterization of forking and dividing in $T_n$.

\begin{theorem}\label{Tn forking ind}
Suppose $A,B,C\subset\H_n$. Then
$$
\textstyle A\ind^f_C B\Leftrightarrow A\ind^d_C B\Leftrightarrow
\begin{array}{l}
\text{$A\cap B\seq C$ and, for all $\bbar\in B\backslash C$,}\\
\text{$K_n(\bbar/AC)$ implies $K_n(\bbar/C)$.}
\end{array}
$$
\end{theorem}
\begin{proof}
The second equivalence is by Theorem \ref{div ind n}; and dividing implies forking in any theory. Therefore we only need to show $A\ind^d_C B$ implies $A\ind^f_C B$. Suppose $A\nind^f_C B$. Then there is some $D\subset\H_n\backslash BC$ such that $A'\nind^d_C BD$ for any $A'\equiv_{BC}A$. By Lemma \ref{fullext}, let $A'\equiv_{BC} A$ such that $A'\ind^R_{BC} D$. By assumption, we have $A'\nind^d_C BD$.

\noindent\textit{Case 1}: $A'\cap BD\not\seq C$. We have $A'\cap BD\seq BC$ by assumption, so this means there is $b\in (A'\cap B)\backslash C$. But $A'\equiv_{BC}A$ and so $b\in (A\cap B)\backslash C$. Therefore $A\nind^d_C B$, as desired.

\noindent\textit{Case 2}: $A'\cap BD\seq C$. Then, since $A'\nind^d_C BD$, it follows from Theorem \ref{div ind n} that there is $\bbar\in BD\backslash C$ such that $\neg K_n(\bbar/C)$ and $K_n(\bbar/A'C)$. Let $X\seq A'C\bbar$ witness $K_n(\bbar/A'C)$. Note that $X\seq A'BCD$. Moreover, note also that if $X\cap(A'\backslash BC)\neq\emptyset$, then $X\seq BCD$, and so $X$ witnesses $K_n(\bbar/C)$, which is a contradiction. 

Therefore $X\cap(A'\backslash BC)\neq\emptyset$. Then we claim that $X\seq A'BC$. Indeed, otherwise there is $u\in X\cap(A'\backslash BC)$ and $v\in X\cap(D\backslash A'BC)$. Therefore $u\neq v$, $u\in A'$, and $v\in\bbar$, and so $u\pR v$, since $X$ witnesses $K_n(\bbar/A'C)$. But this contradicts that there is no edge from $A'\backslash BC$ to $D\backslash BC$. 

So we have $X\seq A'BC$. Let $\bbar_*=X\cap\bbar\in B\backslash C$. Then $\neg K_n(\bbar/C)$ implies $\neg K_n(\bbar_*/C)$, and $X$ witnesses $K_n(\bbar_*/AC)$. Therefore $A'\nind^d_C B$. Since $A'\equiv_{BC} A$, we have $A\nind^d_C B$, as desired.
\end{proof}

It is a general fact that if $\ind^d=\ind^f$ in some theory $T$, then all sets are extension bases for nonforking. Indeed, if a partial type forks over $C$ then it can be extended to a complete type that forks (and therefore divides over $C$). Therefore, by Proposition \ref{facts}$(b)$, no partial type forks over its own set of parameters.

\section{A forking and nondividing formula in $T_n$}\label{example sec}

We have shown that forking and dividing are the same for complete types in $T_n$. In this section, we show that the same result cannot be obtained for partial types, by demonstrating an example of a formula in $T_n$ that forks, but does not divide.

\begin{lemma}\label{4 indisc}
Suppose $(\bbar^l)_{l<\omega}$ is an indiscernible sequence in $\G$ such that $l(\bbar^0)=4$ and $\bbar^0$ is $K_2$-free (no edges). Then either there are $i<j$ such that $\{b^l_i,b^l_j:l<\omega\}$ is $K_2$-free, or $\bigcup_{l<\omega}\bbar^l$ is not $K_3$-free.
\end{lemma}
\begin{proof}
Let $B=\bigcup_{l<\omega}\bbar^l$. Suppose first that for all $i<j$, we have $b^0_i\pR b^1_j$ or $b^1_i\pR b^0_j$. Let 
$$
f:\{(i,j):1\leq i<j\leq 4\}\func\{0,1\}\text{ such that }f(i,j)=0\Leftrightarrow b^0_iRb^1_j.
$$
\textit{Claim}: If $B$ is $K_3$-free then for all $i<j<k$, $f(i,j)=f(j,k)$.\\
\textit{Proof}: Suppose not.\\
\textit{Case 1}: $f(i,j)=1$ and $f(j,k)=0$. If $f(i,k)=0$ then by indiscernibility we have $b^1_i\pR b^0_j$, $b^0_j\pR b^2_k$ and $b^1_i\pR b^2_k$; and so $\{b^1_i,b^0_j,b^2_k\}\cong K_3$. If $f(i,k)=1$ then by indiscernibility we have $b^2_i\pR b^0_j$, $b^0_j\pR b^1_k$ and $b^2_i\pR b^1_k$; and so $\{b^2_i,b^0_j,b^1_k\}\cong K_3$.\\
\textit{Case 2}: $f(i,j)=0$ and $f(j,k)=1$. If $f(i,k)=0$ then $\{b^0_i,b^2_j,b^1_k\}\cong K_3$. Otherwise, if $f(i,k)=1$ then $\{b^1_i,b^2_j,b^0_k\}\cong K_3$. \claim\\
\indent By the claim, if $f(1,2)=0$ then $f(2,3)$, $f(3,4)$, $f(1,3)$ and $f(2,4)$ are all 0; and so $\{b^0_1,b^1_2,b^2_3\}\cong K_3$. If $f(1,2)=1$ then $f(2,3)$, $f(3,4)$, $f(1,3)$, and $f(2,4)$ are all 1; and so $\{b^2_1,b^1_2,b^0_3\}\cong K_3$. In any case we have shown that $B$ is not $K_3$-free.\\
\indent So we may assume that there is some $i<j$ such that $\neg b^0_i\pR b^1_j$ and $\neg b^1_i\pR b^0_j$. By indiscernibility, and since $\neg b^0_i\pR b^0_j$, it follows that $\{b^l_i,b^l_j:l<\omega\}$ is $K_2$-free.
\end{proof}

\begin{theorem}\label{fork nondivide example}
Let $C,\bbar=(b_1,b_2,b_3,b_4)\subset\H_n$ such that $C\cong K_{n-3}$, $\bbar \pR C$, and $\bbar$ is $K_2$-free. For $i<j$, let $\vphi_{i,j}(x,b_j,b_j)=``x\pR Cb_ib_j"$. Set $\vphi(x,\bbar)=\bigvee_{i<j}\vphi_{i,j}(x,b_i,b_j)$. Then $\vphi(x,\bbar)$ forks over $C$ but does not divide over $C$.
\end{theorem}
\begin{proof}
For any $i\neq j$, $|Cb_ib_j|=n-1$, and so $\neg K_n(b_i,b_j/C)$. Moreover, we clearly have that if $a\models\vphi(x,b_i,b_j)$ then $K_n(b_i,b_j/Ca)$. By Theorem \ref{Tn dividing}, $\vphi_{i,j}(x,b_i,b_j)$ divides over $C$, and therefore $\vphi(x,\bbar)$ forks over $C$.

Let $(\bbar^l)_{l<\omega}$ be $C$-indiscernible, with $\bbar^0=\bbar$. If there is some $K_3\cong W\subset\bigcup_{l<\omega}\bbar^l$ then $K_n\cong CW$, since $\bbar^l\equiv_C\bbar$ for all $l<\omega$ implies $C\pR W$. Therefore $\bigcup_{l<\omega}\bbar^l$ is $K_3$-free and so, by Lemma \ref{4 indisc}, there are $i<j$ such that $B:=\{b^l_i,b^l_j:l<\omega\}$ is $K_2$-free. Since $|C|=n-3$, it follows that $BC$ is $K_{n-1}$-free and so there is some $a\in\H$ such that $a\pR (BC)$. But then $a\models\{\vphi_{i,j}(x,b^l_i,b^l_j):l<\omega\}\seq\{\vphi(x,\bbar^l):l<\omega\}$. By Theorem \ref{facts}, $\vphi(x,\bbar)$ does not divide over $A$.
\end{proof}

\section{Final remarks}\label{fin} 

 We have shown that $T_n$ is an $\NSOP_4$ theory in which all sets are extension bases for nonforking, but forking and dividing are not always the same. This only partially answers the question of how the results of \cite{ChKa} extend to theories with $\TP_2$. In partictular, forking is the same as dividing for complete types in $T_n$, which means there is good behavior of nonforking beyond just the fact that all sets are extension bases. This leads to the following amended version of the main question.

\begin{question} Suppose that in some theory all sets are extension bases for nonforking.
\begin{enumerate}
\item Does $\NSOP_3$ imply forking and dividing are the same for partial types? 
\item For what classes of theories do we have $\ind^f=\ind^d$?
\end{enumerate}
\end{question}

\bibliography{D:/Gabe/UIC/gconant/Math/BibTeX/biblio}
\bibliographystyle{amsplain}
\end{document}